\documentclass[a4paper,english,12pt]{article}

\usepackage{amsmath,amsfonts,amssymb,amsthm,bbm,graphics,epsfig,psfrag,epstopdf,dsfont,esint,mathtools,yhmath,geometry,paralist,subfigure,float,color,soul,hyperref}
\usepackage[active]{srcltx}
\usepackage[dvipsnames]{xcolor}

\setstcolor{red}
\newtheorem{theorem}{Theorem}[section]

\newtheorem{lemma}[theorem]{Lemma}

\newtheorem{remark}[theorem]{Remark}

\newcommand{\R}{\mathbb{R}}
\def\N{\mathbb{N}}

\def\epsilon{\varepsilon}

\def\tilde{\widetilde}
\def\div{\mathrm{div}\,}
\def\curl{\mathrm{curl}\,}

\newcommand{\be}{\begin{equation}}
\newcommand{\ee}{\end{equation}}
\newcommand{\baa}{\begin{array}}
\newcommand{\eaa}{\end{array}}
\newcommand{\ba}{\begin{eqnarray}}
\newcommand{\ea}{\end{eqnarray}}
\numberwithin{equation}{section}

\begin{document}
\date{}
\title{\bf{Potential flows away from stagnation\\ in infinite cylinders}
\thanks{This work has been supported by the Excellence Initiative of Aix-Marseille Universit\'e~-~A*MIDEX, a French ``Investissements d'Avenir'' programme, and by the French Agence Nationale de la Recherche (ANR), in the framework of the RESISTE (ANR-18-CE45-0019) and ReaCh (ANR-23-CE40-0023-02) projects. The first author is grateful to the hospitality of the University of Edinburgh, where part of this work was done. The work of the second author was partially supported by EPSRC grant EP/S03157X/1 ``Mean curvature measure of free boundary".}}
\author{Fran\c cois Hamel$^{\hbox{\small{ a}}}$ and Aram Karakhanyan$^{\hbox{\small{ b }}}$\\
\\
\footnotesize{$^{\hbox{a }}$Aix Marseille Univ, CNRS, I2M, Marseille, France} \\
\footnotesize{$^{\hbox{b }}$School of Mathematics, University of Edinburgh,}\\
\footnotesize{James Clerk Maxwell Building, The King’s Buildings, Peter Guthrie Tait Road, Edinburgh EH9 3FD, UK}\\
}
\maketitle
	
\begin{abstract}
\noindent{}Steady incompressible potential flows of an inviscid or viscous fluid are considered in infinite $N$-dimensional cylinders with tangential boundary conditions in any dimension $N\ge2$. We show that such flows, if away from stagnation, are constant and parallel to the direction of the cylinder. This means equivalently that a harmonic function whose gradient is bounded away from zero in an infinite cylinder with Neumann boundary conditions is an affine function. The proof of this rigidity result uses a combination of ODE and PDE arguments, respectively for the streamlines of the flow and the harmonic potential function. 
\end{abstract}


\section{Introduction and main results}\label{intro}

We are firstly concerned with steady incompressible flows solving the Euler equations
\be\label{euler}\left\{\baa{rcll}
v\cdot\nabla\,v+\nabla p & = & 0 & \hbox{in }\Omega,\vspace{3pt}\\
\div v & = & 0 & \hbox{in }\Omega,\vspace{3pt}\\
v\cdot n & = & 0 & \hbox{on }\partial\Omega,\eaa\right.
\ee
in an infinite cylindrical domain
$$\Omega=\R\times\omega\subset\R^N$$
of any dimension $N\ge2$, with bounded connected open section $\omega\subset\R^{N-1}$ of class $C^1$ and outward unit normal $n$. The main results also apply to the steady incompressible Navier-Stokes equations
\be\label{ns}\left\{\baa{rcll}
\mu\,\Delta v+v\cdot\nabla\,v+\nabla p & = & 0 & \hbox{in }\Omega,\vspace{3pt}\\
\div v & = & 0 & \hbox{in }\Omega,\vspace{3pt}\\
v\cdot n & = & 0 & \hbox{on }\partial\Omega,\eaa\right.
\ee
with kinematic viscosity $\mu>0$. The unknown pressure $p$ is assumed to be of class $C^1(\Omega,\R)$ and the unknown velocity field $v$ is assumed to be of class $C^1(\Omega,\R^N)\cap C(\overline{\Omega},\R^N)$. The Navier-Stokes equations are then a priori understood in a weak sense, the velocity field $v$ belonging to $W^{2,p}_{loc}(\Omega,\R^N)$ for all $p\in[1,\infty)$, from standard elliptic estimates.

\subsubsection*{Main results}

We consider potential flows, that is, there exists a $C^2(\Omega,\R)\cap C^1(\overline{\Omega},\R)$ function $\psi$ such that
\be\label{hyp1}
v=\nabla\psi\ \hbox{ in $\overline{\Omega}$},
\ee
and we assume that the flows are away from stagnation, that is,
\be\label{hyp2}
\inf_{\overline{\Omega}}|v|>0.
\ee
Throughout the paper, we call $x\mapsto|x|$ the Euclidean norm in $\R^N$, and $(x,y)\mapsto x\cdot y$ the Euclidean inner product in $\R^N$. We refer the reader to the books \cite{CM,LR,MP} for more discussion on the Euler and the Navier-Stokes equations, as well as various applications. 
 
\begin{theorem}\label{th1}
Let $(v,p)$ be a solution of the steady incompressible Euler or Navier-Stokes equations~\eqref{euler} or~\eqref{ns}, such that $v$ is a potential flow away from stagnation. Then $v$ is constant and parallel to the cylinder. More precisely, there is $a\in\R^*$ such that
\be\label{parallel}
v=(a,0,\cdots,0)\hbox{ in $\overline{\Omega}$},
\ee
and $p$ is constant in $\overline{\Omega}$.
\end{theorem}

For potential flows, the incompressible Euler and Navier-Stokes equations~\eqref{euler} and~\eqref{ns} are equivalent, since $\Delta\partial_{x_i}\psi=0$ in $\Omega$ for all $i\in\{1,\cdots,N\}$ if $\Delta\psi=0$ in~$\Omega$. Actually, the Navier-Stokes equations are usually supplemented with no-slip boundary conditions, for which the normal part of the fluid velocity at a boundary point vanishes (as it is also usually assumed for the Euler equations) and its tangential part equals the velocity of the solid boundary point. Therefore, for an incompressible potential flow away from stagnation and obeying the Navier-Stokes equations in an infinite cylinder with no-slip boundary conditions, Theorem~\ref{th1} shows that the tangential part of the velocity at the boundary is necessarily constant, non-zero, and parallel to the direction of the cylinder. For the Navier-Stokes equations~\eqref{ns}, the case of vanishing tangential part of the velocity at the boundary is incompatible with~\eqref{hyp2}.  However, even with no-slip boundary conditions and vanishing tangential part of the velocity, in the regime of small kinematic viscosities $\mu$, there typically exist boundary layers where the tangential components of the fluid velocity vary rapidly from~$0$ to non-zero values, whereas the normal component does not vary much, and thus stays close to $0$. In that spirit, the Navier-Stokes system~\eqref{ns} with the hypothesis~\eqref{hyp2} can be viewed as an approximation of a model for a viscous flow away from stagnation away from a thin boundary layer.

As a matter of fact, Theorem~\ref{th1} is not only valid for the Euler or Navier-Stokes equations, even if these two systems are the motivation of our work: Theorem~\ref{th1} is based on the structural assumptions~\eqref{hyp1}-\eqref{hyp2} on the velocity field $v$, in addition to its incompressibility and the tangential boundary conditions. Namely, the following theorem holds.

\begin{theorem}\label{th2}
Let $v$ be an incompressible $C^1(\Omega,\R^N)\cap C(\overline{\Omega},\R^N)$ vector field satisfying $v\cdot n=0$ on~$\partial\Omega$ and~\eqref{hyp1}-\eqref{hyp2}. Then there is $a\in\R^*$ such that~\eqref{parallel} holds.
\end{theorem}

Theorem~\ref{th2} clearly implies Theorem~\ref{th1}. Conversely, if $v$ is an incompressible $C^1(\Omega,\R^N)\cap C(\overline{\Omega},\R^N)$ vector field satisfying~\eqref{hyp1} and $v\cdot n=0$ on~$\partial\Omega$, then $v$ obeys the Euler and Navier-Stokes equations~\eqref{euler} and~\eqref{ns} with pressure $p=-|v|^2/2=-|\nabla\psi|^2/2$ up to additive constants (the Bernoulli function $p+|v|^2/2$ is constant). Thus, Theorems~\ref{th1} and~\ref{th2} are equivalent, and are therefore equivalent to the following result of independent interest reformulated in terms of harmonic functions (despite the simplicity of the statement, we are not aware of such a result in the literature, up to our knowledge):

\begin{theorem}\label{th3}
Let $\psi$ be a $C^2(\Omega,\R)\cap C^1(\overline{\Omega},\R)$ harmonic function\footnote{Notice that from its harmonicity, the function $\psi$ is actually real-analytic in~$\Omega$.} such that $\nabla\psi\cdot n=0$ on $\partial\Omega$ and $\inf_{\overline{\Omega}}|\nabla\psi|>0$. Then there are $a\in\R^*$ and $b\in\R$ such that
$$\psi(x)=a\,x_1+b\ \hbox{ for all }x=(x_1,\cdots,x_N)\in\overline{\Omega}.$$
\end{theorem}

We point out that Theorems~\ref{th1}-\ref{th3} hold in any dimension $N\ge 2$. In dimension $N=2$, it turns out that the conclusions of Theorems~\ref{th1}-\ref{th3} could also be derived independently from~\cite[Theorem~1.1]{hn1}. Indeed, first of all, under the assumptions of Theorems~\ref{th1}-\ref{th3}, in dimension $N=2$, the vector field $v$ is of class $C^\infty(\overline{\Omega},\R^2)$, see Section~\ref{sec4}. Then, even without the condition~\eqref{hyp1},~\cite[Theorem~1.1]{hn1} implies that $v$ is a parallel flow, that is, $v(x_1,x_2)=(v_1(x_2),0)$ for all $(x_1,x_2)\in\overline{\Omega}$, and the assumption~\eqref{hyp1} finally entails that $v$ is constant. As a matter of fact, for the specific case of dimension $N=2$, we provide in Section~\ref{sec4} an alternate proof of Theorems~\ref{th1}-\ref{th3}, which does not use~\cite{hn1} and is completely different from the proof given in Sections~\ref{sec2}-\ref{sec3} below for the general case $N\ge 2$. 

In any dimension $N\ge 2$, if $\omega$ is simply connected, that is, if $\Omega$ is simply connected, the Schwarz theorem and Poincar\'e lemma entail that the assumption~\eqref{hyp1} is equivalent to say that the vorticity of the flow $v$ (the antisymmetric part of the Jacobian matrix of~$v$) identically vanishes, that is, the velocity field is irrotational. Therefore, in a simply connected infinite cylinder~$\Omega$ in dimension $N=3$, Theorems~\ref{th1}-\ref{th2} equivalently mean that any incompressible $C^1(\Omega,\R^3)\cap C(\overline{\Omega},\R^3)$ vector field $v$ away from stagnation satisfying $v\cdot n=0$ on~$\partial\Omega$ and $\curl v=(0,0,0)$ is necessarily identically equal to $(a,0,0)$ in $\overline{\Omega}$ for some $a\in\R^*$.

Lastly, an interesting corollary of the main results is that any unbounded potential flow $v\in C^1(\Omega,\R^N)\cap C(\overline{\Omega},\R^N)$ of the Euler or Navier-Stokes equations~\eqref{euler} or~\eqref{ns}, or more generally any unbounded incompressible potential vector field $v\in C^1(\Omega,\R^N)\cap C(\overline{\Omega},\R^N)$ satisfying $v\cdot n=0$ on~$\partial\Omega$, must have a stagnation point in $\overline{\Omega}$ or at infinity, in the sense that $\inf_{\overline{\Omega}}|v|=0$.

\subsubsection*{Counter-examples without the main assumptions}

The conclusions of Theorems~\ref{th1}-\ref{th3} do not hold in general without the assumptions~\eqref{hyp1}-\eqref{hyp2} or without the incompressibility condition $\div v=0$. Let us explain why in more details in the following paragraphs.

First of all, one can not get rid of the condition~\eqref{hyp2} on $\inf_{\overline{\Omega}}|v|>0$ in Theorems~\ref{th1}-\ref{th2} (notice nevertheless that $v$ is never assumed to be bounded, even if the conclusion implies that it is~so!) or the condition $\inf_{\overline{\Omega}}|\nabla\psi|>0$ in Theorem~\ref{th3}. Assume for instance that~$\omega$ is of class $C^2$ and consider any eigenfunction $\phi\in C^\infty(\omega,\R)\cap C^1(\overline{\omega},\R)$ of the Laplace operator in $\omega$ with Neumann boundary conditions, that is, $\Delta\phi+\lambda\phi=0$ in $\omega$ with~$\nabla\phi\cdot n=0$ on~$\partial\omega$, associated with a positive eigenvalue $\lambda>0$ (notice that the function $\phi$ is therefore non-constant and that the $C^2$ smoothness of $\omega$ guarantees that $\phi$ is necessarily of class~$W^{2,p}(\omega,\R)$ for all $p\in[1,\infty)$, and thus at least of class $C^1$ up to the boundary). Setting
$$\psi(x)=\psi(x_1,\cdots,x_N):=\cosh\big(x_1\sqrt{\lambda}\big)\,\phi(x_2,\cdots,x_N),$$
the vector field $v:=\nabla\psi$ is then a non-constant $C^\infty(\Omega,\R^N)\cap C(\overline{\Omega},\R^N)$ solution of~\eqref{euler}-\eqref{ns}, with pressure $p:=-|v|^2/2$, such that $|v(0,x')|=0$ for any critical point $x'\in\overline{\omega}$ of $\phi$ (such a critical point $x'\in\overline{\omega}$ necessarily exists, take for instance a maximal point $x'$ of $\phi$ in $\overline{\omega}$: if $x'\in\omega$, then $|\nabla\phi(x')|=0$ and if $x'\in\partial\omega$, then the tangential derivatives of $\phi$ at $x'$ are all zero, as is the normal derivative by definition of $\phi$, hence $|\nabla\phi(x')|=0$). We point out that the condition $\inf_{\overline{\Omega}}|v|>0$ can not be relaxed into $|v|>0$ in $\overline{\Omega}$ either: for instance, in $\Omega:=\R\times(0,1)$, the vector field $v:=\nabla\psi$, with potential $\psi(x_1,x_2):=e^{\pi x_1}\cos(\pi x_2)$, is a non-constant $C^\infty(\overline{\Omega},\R^2)$ solution of~\eqref{euler}-\eqref{ns}, with pressure $p:=-|v|^2/2$, such that $\inf_{\overline{\Omega}}|v|=0$ and $|v|>0$ in $\overline{\Omega}$. 

Similarly, the conclusions of Theorems~\ref{th1}-\ref{th2} clearly do not hold in general without the incompressibility condition $\div v=0$, and Theorem~\ref{th3} clearly does not hold in general without the harmonicity of $\psi$. For instance, for any $\alpha\in(-1,1)$, the vector field $v:=\nabla\psi$, with
$$\psi(x_1,\cdots,x_N):=x_1+\alpha\ln(x_1^2+1),$$
is a non-incompressible and non-constant $C^\infty(\overline{\Omega},\R^N)$ solution of the Euler equations $v\cdot\nabla v+\nabla p=0$ in $\overline{\Omega}$ and $v\cdot n=0$ on $\partial\Omega$, with pressure $p:=-|v|^2/2$, such that $\inf_{\overline{\Omega}}|v|=1-|\alpha|>0$, that is,~\eqref{hyp1}-\eqref{hyp2} hold. Furthermore, it is easy to cook up to some counter-examples which depend on the variables $(x_2,\cdots,x_N)$ as well. For instance, assume that $\omega$ is of class $C^2$, and consider a bounded and locally H\"older-continuous function $f$ with zero average over $\omega$, and let $\phi$ be the unique $C^2(\omega)\cap C^1(\overline{\omega})$ solution of $\Delta\phi=f$ in~$\omega$ with $\nabla\phi\cdot n=0$ on $\partial\omega$ and zero average over $\omega$. Then, for all $\epsilon>0$ small enough, the vector field $v_\varepsilon:=\nabla\psi_\varepsilon$ with
$$\psi_\varepsilon(x):=x_1+\alpha\ln(x_1^2+1)+\varepsilon\phi(x_2,\cdots,x_N)$$
is a non-incompressible and non-constant $C^1(\Omega,\R^N)\cap C(\overline{\Omega},\R^N)$ solution of the Euler equations $v_\varepsilon\cdot\nabla v_\varepsilon+\nabla p_\varepsilon=0$ in $\Omega$ and $v_\varepsilon\cdot n=0$ on $\partial\Omega$, with pressure $p_\varepsilon:=-|v_\varepsilon|^2/2$, such that $\inf_{\overline{\Omega}}|v_\varepsilon|>0$. However, as follows from Section~\ref{sec2} below, we point out that, for any function $\varphi\in C^2(\Omega,\R)\cap C^1(\overline{\Omega},\R)$ such that $\nabla\varphi\cdot n=0$ on $\partial\Omega$ and $\inf_{\overline{\Omega}}|\nabla\varphi|>0$, then necessarily either $\varphi(x)\to\pm\infty$ as $x_1\to\pm\infty$, or $\varphi(x)\to\mp\infty$ as $x_1\to\pm\infty$, uniformly with respect to $(x_2,\cdots,x_N)\in\overline{\omega}$, even if $\varphi$ is not harmonic (that is, even if the vector field $\nabla\varphi$ is not divergence-free).

The assumption~\eqref{hyp1} saying that $v=\nabla\psi$ is a potential flow can not be relaxed either. For instance, say in $\Omega:=\R\times(0,1)$ in dimension $N=2$, for any non-vanishing and non-constant function $v_1\in C^1((0,1),\R)\cap C([0,1],\R)$, the incompressible non-constant parallel flow
\be\label{parallel2}
v(x_1,x_2):=(v_1(x_2),0)\ee
is a $C^1(\Omega,\R^2)\cap C(\overline{\Omega},\R^2)$ solution of~\eqref{euler} with constant pressure, that fulfills~\eqref{hyp2} but not~\eqref{hyp1} (its scalar vorticity is not identically zero, as $v_1$ is not constant). Let us provide another counter-example, in dimension $N=3$, with $\Omega:=\R\times D$, where $D$ is the unit Euclidean disk of $\R^2$ centered at the origin. The incompressible helicoidal flow $v$ defined by
$$v(x_1,x_2,x_3):=(1,-x_3,x_2)$$
is a $C^\infty(\overline{\Omega},\R^3)$ solution of~\eqref{euler}-\eqref{ns} with pressure $p(x_1,x_2,x_3):=(x_2^2+x_3^2)/2$, that fulfills~\eqref{hyp2} but not~\eqref{hyp1} (its vorticity is equal to the non-zero vector field $(2,0,0)$), and~$v$ is not constant and it is not parallel to the direction $x_1$ either. 

\begin{remark}
We point out that the conclusions of Theorems~$\ref{th1}$-$\ref{th3}$ do not hold in general if the cylinder $\Omega=\R\times\omega$ has an unbounded section~$\omega$. For instance, consider the case of the whole space $\Omega=\R^N$ with $N\ge3$, and the vector field $v:=\nabla\psi\in C^\infty(\R^N,\R^N)$, with
$$\psi(x):=x_1x_2+x_3$$
for all $x\in\R^N$. The vector field $v$ solves~\eqref{euler}-\eqref{ns} with pressure $p:=-|v|^2/2$, it satisfies~\eqref{hyp1}-\eqref{hyp2}, but it is not constant, and is even not a parallel flow. This example holds in dimensions $N\ge3$. In $\R^2$, the vector field $v:=\nabla\psi$ with $\psi(x):=x_1x_2$ does not fulfill~\eqref{hyp2}. Actually, in dimension $N=2$, if a $C^1(\R^2,\R^2)$ vector field satisfies~\eqref{euler} or~\eqref{ns} together with~\eqref{hyp1}-\eqref{hyp2} $($then it is automatically of class $C^\infty(\R^2,\R^2)$ since each of its components is harmonic$)$ and if it is further assumed to be bounded, namely $v\in L^\infty(\R^2,\R^2)$, then it follows from~{\rm{\cite[{\it{Theorem}}~$1.1$]{hn3}}} that~$v$ is a shear flow, that is, there is a $C^\infty(\R,\R)$ function $V$ and a unit vector $e=(e_1,e_2)$ such that $v(x)=V(x\cdot e^\perp)\,e$ for all $x\in\R^2$, with~$e^\perp:=(-e_2,e_1)$, and condition~\eqref{hyp1} finally gives that $v$ is constant.
\end{remark}

\subsubsection*{Comments on rigidity and Liouville-type results for the Euler or Navier-Stokes equations in the literature}

In infinite cylinders, the first rigidity result showed that the solutions $v=(v_1,v_2)$ of~\eqref{euler} in the two-dimensional strip $\Omega=\R\times(0,1)$ satisfying $v_1\neq0$ are necessarily parallel flows of the type~\eqref{parallel2}~\cite{k}. The same conclusion~\eqref{parallel2} also holds for the Euler equations~\eqref{euler} under the non-stagnation condition~\eqref{hyp2}~\cite{hn1}, and for the Boussinesq equations~\cite{cdg} and the hydrostatic Euler equations~\cite{lwx}. The uniqueness of Poiseuille flows such that $|v|=0$ on~$\partial\Omega=\R\times\{0,1\}$, $v(x_1,x_2)\to(x_2(1-x_2),0)$ as $x_1\to-\infty$ and $v_1>0$ in $\Omega=\R\times(0,1)$ was proven in~\cite{llsx}, namely such flows are identically equal to $(x_2(1-x_2),0)$. The uniqueness of Euler flows satisfying a prescribed upstream value as $x_1\to-\infty$ is valid more generally in non-straight infinite two-dimensional nozzles~\cite{llsx} (but non-empty stagnation regions can occur without the condition $v_1>0$~\cite{llsx}) and in three-dimensional axisymmetric nozzles~\cite{dl}. Rigidity and non-rigidity results for shear flows in two-dimensional domains have been obtained for the steady or unsteady Euler and Navier-Stokes equations~\cite{czew2,gxx} (for further stability results of Euler or Navier-Stokes shear flows with monotone profiles, see e.g.~\cite{bgm,bm,czew1,gnrs,ij,ke,z1}). Rigidity results for parallel flows of the Euler equations under the non-stagnation condition~\eqref{hyp2} also hold in two-dimensional half-planes~\cite{hn1,hn2} and in the whole plane~\cite{hn3}. On the other hand, bounded non-parallel flows such that~$|v|>0$ in~$\overline{\Omega}$ but $\inf_{\overline{\Omega}}|v|=0$ exist in the plane, a half-plane or a two-dimensional straight strip~\cite{drr}. Furthermore, if a flow $v$ satisfying some growth condition at infinity in $\R^2$ is not parallel, then the set of its directions $v(x)/|v(x)|$ is necessarily equal to the unit circle~$\mathbb{S}^1$~\cite{gxx}, and this set can also be equal to the half-positive or the half-negative unit circle if the flow is defined in the half-plane $\R\times[0,+\infty)$ or in the strip $\R\times[0,1]$~\cite{gxx}. The rigidity of flows depending only on the radial variable in the class of axisymmetric flows set in hollowed out three-dimensional cylinders $\R\times\omega$ with two-dimensional annular sections $\omega$ has been shown in~\cite{cdg}, together with flexibility results on the existence of non-shear flows in infinite two-dimensional cylinders with approximately straight boundary or non-radially-symmetric flows in hollowed out three-dimensional cylinders with two-dimensional nearly annular sections. Notice that almost all the rigidity results in the references mentioned in this paragraph are concerned with flows in dimension $N=2$, or in some three-dimensional specific domains with symmetries.

In bounded two-dimensional annuli
$$\Omega_{a,b}:=\big\{x\in\R^2:a<|x|<b\big\}$$
with $0<a<b<\infty$, Euler flows have been proved to be circular flows, that is, their streamlines are concentric circles, provided they have no stagnation point in $\overline{\Omega_{a,b}}$~\cite{hn4} or just in~$\Omega_{a,b}$~\cite{wz}. The same radial symmetry conclusion holds in complements of disks with some conditions at infinity and the punctured plane with some conditions at the center and at infinity~\cite{hn4}, and in punctured disks with some conditions at the center~\cite{hn4} or in disks with only one interior stagnation point~\cite{wz} (in a bounded convex domain, the uniqueness of the interior stagnation point is satisfied if the flow is stable in the sense of Arnold~\cite{n1}). Stability results of circular Euler flows in annuli have been shown in~\cite{z2}. Lastly, the structure of the set of flows without stagnation point and whose vorticity has no critical point in general doubly connected sets has been analysed in~\cite{cs}. 

For the time-dependent version of the Navier-Stokes equations in the whole plane $\R^2$, it has been proved in~\cite{knss} that the ancient solutions defined $\{(t,x):t<0,\, x\in\R^2\}$ with bounded velocity are necessarily independent of $x$. The same conclusion holds for the bounded axisymmetric solutions without swirl in $(-\infty,0)\times\R^3$~\cite{knss}. In particular, solutions of~\eqref{ns} with bounded vorticity in $\R^2$, or bounded axisymmetric solutions without swirl in~$\R^3$, are necessarily constant. In the slab $\Omega=\R^2\times(0,1)$, then the solutions of the steady Navier-Stokes equations~\eqref{ns} with no-slip boundary conditions $v=0$ on $\partial\Omega$ are trivial if the Dirichlet integral $\int_{|(x_1,x_2)|<R,\, 0<x_3<1}|\nabla v|^2$ does not grow too fast as $R\to+\infty$, or if $|v|$ and its radial component do not grow too fast as $|(x_1,x_2)|\to+\infty$~\cite{bgwx}. In the slab $\R^2\times\mathbb{T}$ with periodic boundary conditions in the variable $x_3$, the solutions of the steady Navier-Stokes equations are constant parallel to the $x_3$-axis if the radial or swirl components are axisymmetric with respect to the $x_3$-axis, or if the radial component decays faster than $1/|(x_1,x_2)|$ as $|(x_1,x_2)|\to+\infty$, and the solutions are also constant if $\|\,|v|\,\|_{L^\infty(\R^2\times\mathbb{T})}<2\pi$~\cite{bgwx}. The triviality of steady Navier-Stokes flows $v$ in $\R^N\setminus\{0\}$ such that $|v(x)|=O(1/|x|)$ in $\R^N\setminus\{0\}$ is shown in~\cite{bglwx} if $N\ge4$. As an example of a result on the asymptotic dynamics for the unsteady Navier-Stokes equations, we just mention here the large-time convergence in self-similar variables to the Oseen's vortices in $\R^2$ for a large class of initial conditions~\cite{gw}.

Rigidity can also refer to situations where the domain $\Omega$ itself is not given a priori, or when the support of a solution defined in $\R^N$ is not known a priori. Rigidity results for the domain $\Omega\subset\R^2$, namely its radial symmetry, with overdetermined boundary conditions (such as $|v|$ constant on each connectivity component of $\Omega$) has been shown for the Euler equations~\eqref{euler}, in doubly connected domains without interior stagnation point, or in simply connected domains with only one interior stagnation,~\cite{hn4,r1,wz} (the conclusion in simply connected domains does not hold in general without the uniqueness of the interior stagnation point~\cite{r2}). The radial symmetry of a compactly supported flow $v\in C^1(\R^2,\R^2)$ solving~\eqref{euler} in $\R^2$ has also been proved if the open set $\{x:|v(x)|>0\}$ is assumed to be doubly connected~\cite{r1} (however non-circular compactly supported Lipschitz-continous or $C^k(\R^2)$ flows whose supports are close to radially symmetric annuli  exist~\cite{efr,efrs}). For the vorticity patch problem for which the, scalar, vorticity is assumed to be the indicator function of a set $D\subset\R^2$, then $D$ is necessarily a disk and the flow is circular~\cite{f,gpsy} (see also~\cite{gpsy,h,hmv} for symmetry results of rotating patches). The flow $v:\R^2\to\R^2$ is also circular if its vorticity is smooth, nonnegative and compactly supported~\cite{gpsy}, but the conclusion is false in general for sign-changing vorticity~\cite{gps}. In dimension $N=3$, Beltrami flows (for which $v\times \curl v=0$) of the Euler equations~\eqref{euler} in $\R^3$ are necessarily $0$ if they have compact support~\cite{n2} or if they have finite energy or decay fast enough at infinity~\cite{cc}. This Liouville-type conclusion also holds for axisymmetric flows with finite energy, no swirl, trivial limit at infinity and constant pressure at infinity~\cite{jx}. Nevertheless, there still exist non-trivial compactly supported, and axisymmetric, flows in $\R^3$ whose support is a torus with almost circular section~\cite{clv,dep,g}. Actually, the structure of steady Euler flows in $\R^3$ can be in general quite complex, and, even for Beltrami flows, thin vortex tubes of any link and knot type exist in general~\cite{ep1,ep2}.

\subsubsection*{Elements of the proofs and outline of the paper}

The proofs of Theorems~\ref{th1}-\ref{th3} rely first on the study of the geometric properties of the streamlines of a potential flow $v=\nabla\psi$ under the assumption~\eqref{hyp2}. Namely, we show in Section~\ref{sec2}, with ODE-type arguments, that these streamlines all go from one end to the other end of the cylinder $\Omega=\R\times\omega$, and that either $\psi(x_1,\cdot)\to\pm\infty$ as $x_1\to\pm\infty$ uniformly in $\overline{\omega}$, or $\psi(x_1,\cdot)\to\mp\infty$ as $x_1\to\pm\infty$ uniformly in $\overline{\omega}$. Section~\ref{sec3} is based on PDE-type arguments, comparison with suitable auxiliary affine functions, maximum principles and the Harnack inequality, using this time the incompressibility condition of $v$, that is, the harmonicity of $\psi$. We underline that the proofs hold in any dimension $N$ and do not use any symmetry or topological properties (such as simple or double connectivity) of the section $\omega$. However, we provide in Section~\ref{sec4} a completely different proof based on PDE arguments only, in the particular case of dimension $N=2$.


\section{Properties of the streamlines}\label{sec2}

We assume here that $v=\nabla\psi$ is a $C^1(\Omega,\R^N)\cap C(\overline{\Omega},\R^N)$ vector field satisfying $v\cdot n=0$ on~$\partial\Omega$ and $\inf_{\overline{\Omega}}|v|>0$ (but here we do not assume the incompressibility condition). For~$x\in\overline{\Omega}$, we denote $\xi_x$ the solution of
$$\left\{\baa{rcl}
\xi_x'(\tau) & = & v(\xi_x(\tau)),\vspace{3pt}\\
\xi_x(0) & = & x.\eaa\right.$$
From the Cauchy-Lipschitz theorem and the tangential boundary condition $v\cdot n=0$ on $\partial\Omega$, each $\xi_x$ is defined in a maximal open interval $I_x:=(\tau^-_x,\tau^+_x)\subset\R$ with $-\infty\le\tau^-_x<0<\tau^+_x\le+\infty$. Furthermore, the streamlines
$$\Xi_x:=\xi_x(I_x)$$
are pairwise disjoint, and $\Xi_x\subset\Omega$ (resp. $\Xi_x\subset\partial\Omega$) if $x\in\Omega$ (resp. $x\in\partial\Omega$). It is also known that the streamlines are analytic curves~\cite{kn}.

The first auxiliary lemma shows the divergence to infinity of every streamline~$\Xi_x$ at its ends.

\begin{lemma}\label{lem1}
Let $v=\nabla\psi$ be a $C^1(\Omega,\R^N)\cap C(\overline{\Omega},\R^N)$ vector field satisfying $v\cdot n=0$ on~$\partial\Omega$ and
$$\eta:=\inf_{\overline{\Omega}}|v|>0.$$
Then, for each $x\in\overline{\Omega}$,
$$|\xi_x(\tau)|\to+\infty\hbox{ as $\tau\to\tau^-_x$ and as $\tau\to\tau^+_x$},$$
and $\psi(\xi_x(\tau))\to\pm\infty$ as $\tau\to\tau^\pm_x$.
\end{lemma}

\begin{proof}
Consider any $x$ in $\overline{\Omega}$. Let us prove that $|\xi_x(\tau)|\to+\infty$ as $\tau\to\tau^-_x$ (the limit as~$\tau\to\tau^+_x$ can be handled similarly). Assume by way of contradiction that the conclusion does not hold. There exist then a sequence $(\tau_n)_{n\in\N}$ in $I_x=(\tau^-_x,\tau^+_x)$ and a point $\xi\in\overline{\Omega}$ such that
$$\tau_n\to\tau^-_x\ \hbox{ and }\ \xi_x(\tau_n)\to\xi\ \hbox{ as $n\to+\infty$}.$$
Observe that the real-valued function $g_x$ defined in $I_x$ by
\be\label{defgx}
g_x(\tau):=\psi(\xi_x(\tau))
\ee
is of class $C^1(I_x)$ and satisfies
$$g_x'(\tau)=\nabla\psi(\xi_x(\tau))\cdot\xi'_x(\tau)=|v(\xi_x(\tau))|^2\ge\eta^2$$
for all $\tau\in I_x$, hence $g_x(\tau)\to-\infty$ as $\tau\to\tau^-_x$ if $\tau^-_x=-\infty$. But since $g_x(\tau_n)=\psi(\xi_x(\tau_n))\to\psi(\xi)$ and $\tau_n\to\tau^-_x$ as $n\to+\infty$, one infers that $\tau^-_x\neq-\infty$, that is,~$\tau^-_x\in(-\infty,0)$. Now, from the Cauchy-Lipschitz theorem, there are $\sigma>0$ and $r>0$ such that
$$[-\sigma,\sigma]\subset I_y\hbox{ for all $y\in\overline{\Omega}$ such that $|y-\xi|<r$}.$$
In particular, $[\tau_n-\sigma,\tau_n+\sigma]\subset I_x$ for all $n$ large enough (for which $|\xi_x(\tau_n)-\xi|<r$), due to the maximality of the interval $I_x=(\tau^-_x,\tau^+_x)$. Hence, $\tau_n-\sigma>\tau^-_x$ for all $n$ large enough, and the limit as $n\to+\infty$ contradicts the finiteness of $\tau^-_x$ and the positivity of $\sigma$. As a consequence,
$$|\xi_x(\tau)|\to+\infty\hbox{ as $\tau\to\tau^-_x$ (and as $\tau\to\tau^+_x$, similarly)}.$$

Now, since $g_x'(\tau)=|v(\xi_x(\tau))|^2\ge\eta\,|\xi_x'(\tau)|$ for every $\tau\in(\tau^-_x,\tau^+_x)$, one gets that
\be\label{psitautau'}\baa{rcl}
\displaystyle\psi(\xi_x(\tau'))-\psi(\xi_x(\tau))=g_x(\tau')-g_x(\tau)\ge\eta\int_{\tau}^{\tau'}|\xi_x'(s)|\,ds & \ge & \displaystyle\eta\ \Big|\!\!\int_\tau^{\tau'}\!\!\xi_x'(s)\,ds\Big|\vspace{3pt}\\
& = & \eta\,|\xi_x(\tau')-\xi_x(\tau)|\eaa
\ee
for every $\tau^-_x<\tau\le\tau'<\tau^+_x$. It then follows from the previous paragraph that
$$\psi(\xi_x(\tau))\to\pm\infty\ \hbox{ as $\tau\to\tau^\pm_x$},$$
completing the proof of Lemma~\ref{lem1}.
\end{proof}

\begin{remark}\label{rem1}
The fact that each function $g_x=\psi\circ\xi_x$ defined by~\eqref{defgx} is increasing in $I_x=(\tau^-_x,\tau^+_x)$ implies that the map $\xi_x:I_x\to\Xi_x$ is one-to-one, for each $x\in\overline{\Omega}$. In particular, each streamline $\Xi_x$ is a simple curve.
\end{remark}

The following lemma shows that the streamlines $\Xi_x$ all go from the same end to the other end of the cylinder $\Omega$, and that $\psi$ converges to opposite infinities at the ends of the cylinders. In the sequel, we call
$$\mathrm{e}_1:=(1,0,\cdots,0)$$
the first vector of the canonical basis of $\R^N$.

\begin{lemma}\label{lem2}
Under the assumptions of Lemma~$\ref{lem1}$, then either
$$\xi_x(\tau)\cdot\mathrm{e}_1\to\pm\infty\hbox{ as $\tau\to\tau^\pm_x$ for all $x\in\overline{\Omega}$},$$
or
$$\xi_x(\tau)\cdot\mathrm{e}_1\to\mp\infty\hbox{ as $\tau\to\tau^\pm_x$ for all $x\in\overline{\Omega}$}.$$
Furthermore, in the former case, then
$$\psi(x_1,\cdot)\to\pm\infty\hbox{ as $x_1\to\pm\infty$ uniformly in $\overline{\omega}$},$$
while
$$\psi(x_1,\cdot)\to\mp\infty\hbox{ as $x_1\to\pm\infty$ uniformly in $\overline{\omega}$}$$
in the latter case.
\end{lemma}

\begin{proof}
For each $x\in\overline{\Omega}$, it follows from Lemma~\ref{lem1} and the continuity of the map $\xi_x:I_x\to\overline{\Omega}$ that either $\xi_x(\tau)\cdot\mathrm{e}_1\to+\infty$ as $\tau\to\tau^\pm_x$, or $\xi_x(\tau)\cdot\mathrm{e}_1\to-\infty$ as $\tau\to\tau^\pm_x$, or $\xi_x(\tau)\cdot\mathrm{e}_1\to\pm\infty$ as $\tau\to\tau^\pm_x$, or $\xi_x(\tau)\cdot\mathrm{e}_1\to\mp\infty$ as $\tau\to\tau^\pm_x$. Therefore, by defining
$$\left\{\baa{l}
\Omega_1:=\big\{x\in\overline{\Omega}:\xi_x(\tau)\cdot\mathrm{e}_1\to+\infty\hbox{ as $\tau\to\tau^\pm_x$}\big\},\vspace{5pt}\\
\Omega_2:=\big\{x\in\overline{\Omega}:\xi_x(\tau)\cdot\mathrm{e}_1\to-\infty\hbox{ as $\tau\to\tau^\pm_x$}\big\},\vspace{5pt}\\
\Omega_3:=\big\{x\in\overline{\Omega}:\xi_x(\tau)\cdot\mathrm{e}_1\to\pm\infty\hbox{ as $\tau\to\tau^\pm_x$}\big\},\vspace{5pt}\\
\Omega_4:=\big\{x\in\overline{\Omega}:\xi_x(\tau)\cdot\mathrm{e}_1\to\mp\infty\hbox{ as $\tau\to\tau^\pm_x$}\big\},\eaa\right.$$
there holds
\be\label{Omegai}
\overline{\Omega}=\Omega_1\cup\Omega_2\cup\Omega_3\cup\Omega_4.
\ee

The sets $\Omega_i$'s (for $i=1,\cdots,4$) are clearly pairwise disjoint. Let us now show that they are all open relatively to $\overline{\Omega}$. It will then follow by connectivity of $\overline{\Omega}$ that $\overline{\Omega}$ will be equal to one of the $\Omega_i$'s, and we will next rule out the cases $\overline{\Omega}=\Omega_1$ and $\overline{\Omega}=\Omega_2$.

Let us first show that the set $\Omega_1$ is open relatively to $\overline{\Omega}$. Let us define
\be\label{defM}
M:=\max_{\overline{\omega}}|\psi(0,\cdot)|.
\ee
Pick any $x$ in $\Omega_1$. From Lemma~\ref{lem1} and the definition of $\Omega_1$, there are $s^-<s^+$ in the interval $(\tau^-_x,\tau^+_x)$ such that
$$\xi_x(s^\pm)\cdot\mathrm{e}_1>0\ \hbox{ and }\ \psi(\xi_x(s^-))<-M\le M<\psi(\xi_x(s^+)).$$
From the Cauchy-Lipschitz theorem, there exists $r>0$ such that, for every $y\in\overline{\Omega}$ with~$|y-x|<r$, one has
$$[s^-,s^+]\subset(\tau^-_y,\tau^+_y),\ \ \xi_y(s^\pm)\cdot\mathrm{e}_1>0,\ \hbox{ and }\ \psi(\xi_y(s^-))<-M\le M<\psi(\xi_y(s^+)).$$
For each such $y$, Remark~\ref{rem1} entails that
$$\left\{\baa{ll}
\psi(\xi_y(\tau))<-M & \hbox{for all $\tau\in(\tau^-_y,s^-]$},\vspace{3pt}\\
\psi(\xi_y(\tau))>M & \hbox{for all $\tau\in[s^+,\tau^+_y)$}.\eaa\right.$$
Together with the definition of $M$ and the continuity of $\xi_y$, one infers that
$$\xi_y(\tau)\cdot\mathrm{e}_1>0\ \hbox{ for all $\tau\in(\tau^-_y,s^-]\cup[s^+,\tau^+_y)$},$$
and Lemma~\ref{lem1} finally gives that $\xi_y(\tau)\cdot\mathrm{e}_1\to+\infty$ as $\tau\to\tau^\pm_y$. In other words, any point~$y$ in~$\overline{\Omega}$ such that $|y-x|<r$ belongs to $\Omega_1$. The set $\Omega_1$ is thus open relatively to $\overline{\Omega}$. Similarly, so is the set $\Omega_2$.

Let us then show that the set $\Omega_3$ is open relatively to $\overline{\Omega}$. Pick any in $\Omega_3$. From Lemma~\ref{lem1} and the definition of $\Omega_3$, there are $\sigma^-<\sigma^+$ in the interval $(\tau^-_x,\tau^+_x)$ such that
$$\xi_x(\sigma^-)\cdot\mathrm{e}_1<0<\xi_x(\sigma^+)\cdot\mathrm{e}_1\ \hbox{ and }\ \psi(\xi_x(\sigma^-))<-M\le M<\psi(\xi_x(\sigma^+)).$$
From the Cauchy-Lipschitz theorem, there exists $\rho>0$ such that, for every $y\in\overline{\Omega}$ with~$|y-x|<\rho$, one has
$$[\sigma^-,\sigma^+]\subset\!(\tau^-_y,\tau^+_y),\ \ \xi_y(\sigma^-)\cdot\mathrm{e}_1<0<\xi_y(\sigma^+)\cdot\mathrm{e}_1$$
and
$$\psi(\xi_y(\sigma^-))<-M\le M<\psi(\xi_y(\sigma^+)).$$
As in the previous paragraph, for each such $y$, one has $\psi(\xi_y(\tau))<-M$ for all $\tau\in(\tau^-_y,\sigma^-]$ and $\psi(\xi_y(\tau))>M$ for all $\tau\in[\sigma^+,\tau^+_y)$, hence
$$\xi_y(\tau)\cdot\mathrm{e}_1<0\hbox{ for all $\tau\in(\tau^-_y,\sigma^-]\ $ and $\ \xi_y(\tau)\cdot\mathrm{e}_1>0$ for all $\tau\in[\sigma^+,\tau^+_y)$}.$$
One gets from Lemma~\ref{lem1} that $\xi_y(\tau)\cdot\mathrm{e}_1\to\pm\infty$ as $\tau\to\tau^\pm_y$, that is, $y\in\Omega_3$. The set $\Omega_3$ is thus open relatively to $\overline{\Omega}$. Similarly, so is the set~$\Omega_4$.

From~\eqref{Omegai} and the connectivity of $\overline{\Omega}$, it follows that there is $i\in\{1,2,3,4\}$ such that
$$\overline{\Omega}=\Omega_i.$$
Assume by way of contradiction that $i=1$. Consider any $x\in\overline{\Omega}$ such that~$x_1=x\cdot\mathrm{e}_1\le0$. Owing to the definition of $\Omega_1$, there are $\theta^-\in(\tau^-_x,0]$ and $\theta^+\in[0,\tau^+_x)$ such that~$\xi_x(\theta^\pm)\cdot\mathrm{e}_1=0$. From the definition of $M$ in~\eqref{defM} and a calculation similar to~\eqref{psitautau'}, one gets that
$$\baa{rcl}
\displaystyle2M\ge\psi(\xi_x(\theta^+))-\psi(\xi_x(\theta^-)) & \ge & \displaystyle\eta\,\Big|\int_{\theta^-}^0\xi_x'(s)\,ds\Big|+\eta\,\Big|\int_0^{\theta^+}\xi_x'(s)\,ds\Big|\vspace{3pt}\\
& = & \eta\times\big(|x-\xi_x(\theta^-)|+|x-\xi_x(\theta^+)|\big)\ge2\,\eta\,|x_1|.\eaa$$
The limit as $x_1\to-\infty$ leads to a contradiction. Therefore, $\overline{\Omega}$ cannot be equal to $\Omega_1$. Similarly, $\overline{\Omega}$ can not be equal to $\Omega_2$. As a consequence,
$$\hbox{either $\overline{\Omega}=\Omega_3\ $ or $\ \overline{\Omega}=\Omega_4$}.$$
This provides the first part of the conclusion of Lemma~\ref{lem2}.

To show the last part of the conclusion, we only consider the case $\overline{\Omega}=\Omega_3$ (the case~$\overline{\Omega}=\Omega_4$ can be handled similarly). One has to show that $\psi(x_1,\cdot)\to\pm\infty$ as $x_1\to\pm\infty$, uniformly in $\overline{\omega}$. For any $x\in\overline{\Omega}$ such that $x_1\ge0$, since we are in the case $\overline{\Omega}=\Omega_3$, there is then $\vartheta^-\in(\tau^-_x,0]$ such that $\xi_x(\vartheta^-)\cdot\mathrm{e}_1=0$. A similar calculation as in the previous paragraph then gives
$$\psi(x)+M\ge\psi(\xi_x(0))-\psi(\xi_x(\vartheta^-))\ge\eta\,|x-\xi_x(\vartheta^-)|\ge\eta\,x_1.$$
Thus, $\psi(x_1,\cdot)\to+\infty$ as $x_1\to+\infty$, uniformly in $\overline{\omega}$. The limit
$$\lim_{x_1\to-\infty}\Big(\max_{\overline{\omega}}\psi(x_1,\cdot)\Big)=-\infty$$
holds similarly, and the proof of Lemma~\ref{lem2} is thereby complete.
\end{proof}


\section{Proof of Theorems~\ref{th1}-\ref{th3}}\label{sec3}

We consider a vector field $v=\nabla\psi$ satisfying $v\cdot n=0$ on~$\partial\Omega$, for a $C^2(\Omega,\R)\cap C^1(\overline{\Omega},\R)$ function $\psi$ such that $\Delta\psi=0$ in $\Omega$ and $\inf_{\overline{\Omega}}|v|=\inf_{\overline{\Omega}}|\nabla\psi|>0$. From Lemma~\ref{lem2}, up to changing $\psi$ and $v$ into $-\psi$ and $-v$, one can assume without loss of generality that
\be\label{psilim}
\psi(x_1,\cdot)\to\pm\infty\ \hbox{ as $x_1\to\pm\infty$, uniformly in $\overline{\omega}$}.
\ee
Denote
$$a:=\liminf_{x_1\to+\infty}\frac{\displaystyle\min_{\overline{\omega}}\psi(x_1,\cdot)}{x_1}\ \in[0,+\infty].$$
We shall show that $a\in(0,+\infty)$ and that there is $b\in\R$ such that $\psi(x)=ax_1+b$ for all $x=(x_1,\cdots,x_N)\in\overline{\Omega}$.

To do so, consider first any $\alpha\in(-\infty,a)$ (thus, $\alpha\in\R$, even if $a$ were equal to $+\infty$). Consider $M\in\R$ as in~\eqref{defM}. Since $\alpha<a$, there is $A>0$ such that
$$\psi(x)\ge -M+\alpha x_1\ \hbox{ for all $x\in[A,+\infty)\times\overline{\omega}$}.$$
Since the function $x\mapsto-M+\alpha x_1$ is harmonic, below $\psi$ on $(\{0\}\times\overline{\omega})\,\cup\,([A,+\infty)\times\overline{\omega})$ and since it satisfies Neumann boundary conditions on $\partial\Omega=\R\times\partial\omega$ (as it is independent of the variables $(x_2,\cdots,x_N)$), the maximum principle and the Hopf lemma entail that
\be\label{ineqpsi}
\psi(x)\ge-M+\alpha x_1\ \hbox{ in $[0,B]\times\overline{\omega}$},
\ee
for every $B\ge A$. Indeed, for $B\ge A$, if the harmonic function $\tilde{\psi}:x\mapsto\psi(x)-(-M+\alpha x_1)$ reaches its minimum in $[0,B]\times\overline{\omega}$ at an interior point $x_m\in(0,B)\times\omega$, then the strong interior maximum principle implies that $\tilde{\psi}$ is constant, hence it is nonnegative as it is so on $\{0,B\}\times\overline{\omega}$; on the other hand, if $\tilde{\psi}$ reaches its minimum at a point~$y_m$ lying on $(0,B)\times\partial\omega$ without having any interior minimum point, then the Hopf lemma implies that $\nabla\tilde{\psi}\cdot n<0$ at $y_m$, a contradiction; therefore, there always holds $\min_{[0,B]\times\overline{\omega}}\tilde{\psi}=\min_{\{0,B\}\times\overline{\omega}}\tilde{\psi}\ge0$, yielding~\eqref{ineqpsi}. Therefore, the limit as $B\to+\infty$ yields
$$\psi(x)\ge-M+\alpha x_1\ \hbox{ for all $x\in[0,+\infty)\times\overline{\omega}$}.$$
Since $\alpha$ is arbitrary in $(-\infty,a)$, it follows that $a\in\R$, thus $0\le a<+\infty$, and
\be\label{ineqinf}
\psi(x)\ge -M+ax_1\ \hbox{ for all $x\in[0,+\infty)\times\overline{\omega}$}.
\ee

Now, since the harmonic function $\phi$ defined by
$$\phi(x):=\psi(x)-(-M+ax_1)$$
satisfies the Neumann boundary conditions on $\partial\Omega$ and is nonnegative in $[0,+\infty)\times\overline{\omega}$, one infers from the Harnack inequality that there is a constant $C\in[1,+\infty)$ such that
$$\phi(x_1,y)\le C\phi(x_1,y')\ \hbox{ for all $x_1\ge1$ and $y,y'\in\overline{\omega}$}.$$
Consider any $\epsilon>0$. Owing to the definition of $a$, there is a sequence $(s_n)_{n\in\N}$ in $[1,+\infty)$ diverging to $+\infty$ such that $\min_{\overline{\omega}}\psi(s_n,\cdot)\le(a+\epsilon)\,s_n$, hence
$$\min_{\overline{\omega}}\phi(s_n,\cdot)\le\epsilon\,s_n+M$$
for all $n\in\N$. Therefore,
$$\max_{\overline{\omega}}\phi(s_n,\cdot)\le C(\epsilon\,s_n+M)$$
for all $n\in\N$, that is,
$$\max_{\overline{\omega}}\psi(s_n,\cdot)\le (a+C\epsilon)\,s_n+(C-1)M.$$
As in the previous paragraph, the maximum principle and the Hopf lemma imply that
$$\psi(x)\le M+(a+C\epsilon)\,x_1+(C-1)M\ \hbox{ for all $n\in\N$ and $x\in[0,s_n]\times\overline{\omega}$}$$
(the additional constant $M$ in the right-hand side guarantees the comparison on $\{0\}\times\overline{\omega}$). Since $s_n\to+\infty$ as $n\to+\infty$ and $\epsilon>0$ was arbitrary, one gets that
$$\psi(x)\le CM+ax_1\ \hbox{ for all $x\in[0,+\infty)\times\overline{\omega}$}.$$
Together with~\eqref{ineqinf}, this means that the function $x\mapsto\psi(x)-ax_1$ is bounded in $[0,+\infty)\times\overline{\omega}$.

Similarly, remembering~\eqref{psilim} and calling
$$a^-:=\liminf_{x_1\to-\infty}\frac{\displaystyle\max_{\overline{\omega}}\psi(x_1,\cdot)}{x_1}\ \in[0,+\infty],$$
one can show that $a^-$ is actually a real number, in $[0,+\infty)$, and that the function $x\mapsto\psi(x)-a^-x_1$ is bounded in $(-\infty,0]\times\overline{\omega}$.

Assume now by way of contradiction that $a^-\neq a$. Assume first that $a^-<a$, and pick any $\beta$ such that $a^-<\beta<a$. From the conclusions of the previous two paragraphs, one knows that
$$\psi(x)-\beta x_1\to+\infty\ \hbox{ as $x_1\to\pm\infty$ uniformly in $\overline{\omega}$}.$$
Therefore, by calling $\gamma$ the minimum value of the function $x\mapsto\psi(x)-\beta x_1$, the function $x\mapsto\psi(x)-\beta x_1-\gamma$ is then nonnegative in $\overline{\Omega}$ and vanishes somewhere in $\overline{\Omega}$. Since it is harmonic in $\Omega$ and satisfies Neumann boundary conditions on $\partial\Omega$, the maximum principle and the Hopf lemma imply that
$$\psi(x)-\beta x_1-\gamma=0\ \hbox{ for all $x\in\overline{\Omega}$},$$
which is clearly impossible since $\psi(x)-\beta x_1\to+\infty$ as $x_1\to\pm\infty$ uniformly in $\overline{\omega}$. As a consequence, $a^-$ can not be less than $a$. Similarly, if $a^-$ were larger than $a$, by putting above $\psi$ some functions of the type $x\mapsto\delta x_1+\rho$ with $\delta\in(a,a^-)$ and some $\rho\in\R$, one gets a similar contradiction. In other words,
$$a^-=a,$$
and the harmonic function $\varphi$ defined in $\overline{\Omega}$ by
$$\varphi(x):=\psi(x)-ax_1,$$
is bounded in $\overline{\Omega}$.\footnote{Another way to conclude that $a^-=a$ would be to integrate the equation $\Delta\psi=0$ over $(s_1,s_2)\times\omega$, for arbitrary $s_1<s_2$ in $\R$, to get that the map $x_1\mapsto\int_\omega\partial_{x_1}\psi(x_1,\cdot)$ is equal to a constant $c\in\R$. Hence, $\int_\omega\psi(x_1,\cdot)/x_1\to c$ as $x_1\to\pm\infty$, implying that $c=a=a^-$, because the functions $x\mapsto\psi(x)-ax_1$ and $x\mapsto\psi(x)-a^-x_1$ are bounded respectively in $[0,+\infty)\times\overline{\omega}$ and $(-\infty,0]\times\overline{\omega}$.}

It is now quite standard to conclude that $\varphi$ is constant in $\overline{\Omega}$. For the sake of completeness, let us briefly sketch one elementary proof. So, let $Y\in\overline{\omega}$ be such that
$$b:=\min_{\overline{\omega}}\varphi(0,\cdot)=\varphi(0,Y).$$
For any $\epsilon>0$, there is $A_\epsilon>0$ such that $\varphi(x)\ge b-\epsilon x_1$ for all $x\in[A_\epsilon,+\infty)\times\overline{\omega}$ and, as above, it follows from the maximum principle and the Hopf lemma that $\varphi(x)\ge b-\epsilon x_1$ for all $B\ge A_\epsilon$ and $x\in[0,B]\times\overline{\omega}$, hence
$$\varphi(x)\ge b\ \hbox{ for all $x\in[0,+\infty)\times\overline{\omega}$}$$
by passing to the limits $B\to+\infty$ and then $\epsilon\to0$. Similarly, one gets that $\varphi(x)\ge b$ for all $x\in(-\infty,0]\times\overline{\omega}$. Therefore,
$$\varphi(x)\ge b=\varphi(0,Y)\ \hbox{ for all $x\in\overline{\Omega}$}.$$
Since $\varphi$ is harmonic in $\Omega$ and satisfies Neumann boundary conditions on $\partial\Omega$, the maximum principle and the Hopf lemma imply that $\varphi$ is constant in $\overline{\Omega}$, that is,
$$\psi(x)=ax_1+b\ \hbox{ for all $x\in\overline{\Omega}$}.$$
As $\inf_{\overline{\Omega}}|\nabla\psi|>0$ by~\eqref{hyp1}-\eqref{hyp2}, the real number $a$ is not zero. Lastly, since $v=\nabla\psi$ is constant, so is $p$ necessarily, and the proof of Theorems~\ref{th1}-\ref{th3} is thereby complete.\hfill$\Box$


\section{An alternate proof of Theorems~\ref{th1}-\ref{th3} in dimension $N=2$}\label{sec4}

Consider in this section a two-dimensional strip
$$\Omega=\R\times(\alpha^-,\alpha^+)\subset\R^2,$$
with $\alpha^-<\alpha^+\in\R$, and a vector field $v=\nabla\psi$ satisfying $v\cdot n=0$ on~$\partial\Omega=\R\times\{\alpha^-,\alpha^+\}$, for a $C^2(\Omega,\R)\cap C^1(\overline{\Omega},\R)$ function $\psi$ such that $\Delta\psi=0$ in $\Omega$ and $\inf_{\overline{\Omega}}|v|=\inf_{\overline{\Omega}}|\nabla\psi|>0$. 

First of all, let $\tilde{\psi}$ be the function defined in $\tilde{\Omega}:=\R\times(2\alpha^--\alpha^+,2\alpha^+-\alpha^-)$ by
$$\tilde{\psi}(x)=\tilde{\psi}(x_1,x_2)=\left\{\baa{ll}
\psi(x_1,2\alpha^--x_2) & \hbox{if }2\alpha^--\alpha^+<x_2<\alpha^-,\vspace{3pt}\\
\psi(x_1,x_2) & \hbox{if }\alpha^-\le x_2\le\alpha^+,\vspace{3pt}\\
\psi(x_1,2\alpha^+-x_2) & \hbox{if }\alpha^+<x_2<2\alpha^+-\alpha^-.\eaa\right.$$
Since $\nabla\psi\cdot n=0$ on $\partial\Omega=\R\times\{\alpha^-,\alpha^+\}$, the function $\tilde{\psi}$ is of class $C^1(\tilde{\Omega},\R)$, and it is also a weak harmonic function, that is, $\int_{\tilde{\Omega}}\nabla\tilde{\psi}\cdot\nabla\varphi=0$ for all $\varphi\in C^1_c(\tilde{\Omega},\R)$. Then $\tilde{\psi}$ is harmonic in $\tilde{\Omega}$, hence $\psi$ is of class $C^\infty(\overline{\Omega},\R)$, and in particular of class $C^2$ up to the boundary~$\partial\Omega$.

Consider now the function $\phi:\overline{\Omega}\to\R$ given by
$$\phi:=|\nabla\psi|^{-2}=\frac{1}{(\partial_{x_1}\psi)^2+(\partial_{x_2}\psi)^2}.$$
From the previous paragraph and the assumption $\inf_{\overline{\Omega}}|\nabla\psi|>0$, $\phi$ is a positive bounded function of class $C^\infty(\overline{\Omega},\R)$. Using the identities
$$\partial_{x_1x_1}\psi+\partial_{x_2x_2}\psi=\Delta\psi=\Delta\partial_{x_1}\psi=\Delta\partial_{x_2}\psi=0\ \hbox{ in $\overline{\Omega}$},$$
one finds after a straightforward calculation that $\phi$ satisfies
\be\label{Deltaphi}
\Delta\phi=2\,\phi^2\,\sum_{1\le i,j\le 2}(\partial_{x_ix_j}\psi)^2\ge0\ \hbox{ in }\overline{\Omega}.
\ee
Furthermore, since $\partial_{x_2}\psi(x_1,\alpha^\pm)=0$ and then $\partial_{x_1x_2}\psi(x_1,\alpha^\pm)=0$ for all $x_1\in\R$, one gets that
$$\baa{rcl}
\partial_{x_2}\phi(x_1,\alpha^\pm) & \!\!\!=\!\!\! & -2\,\phi(x_1,\alpha^\pm)^2\,\big(\partial_{x_1}\psi(x_1,\alpha^\pm)\partial_{x_1x_2}\psi(x_1,\alpha^\pm)+\partial_{x_2}\psi(x_1,\alpha^\pm)\partial_{x_2x_2}\psi(x_1,\alpha^\pm)\big)\vspace{3pt}\\
& \!\!\!=\!\!\! & 0\eaa$$
for all $x_1\in\R$, that is,
\be\label{neumann}
\nabla\phi\cdot n=0\ \hbox{ on $\partial\Omega$}.
\ee
Now, for any $s_1<s_2$ in $\R$, integrating the inequation $\Delta\phi\ge0$ over $(s_1,s_2)\times(\alpha^-,\alpha^+)$ leads to
$$0\le\int\!\!\!\!\int_{(s_1,s_2)\times(\alpha^-,\alpha^+)}\Delta\phi(x_1,x_2)\,dx_1\,dx_2=\int_{\alpha^-}^{\alpha^+}\partial_{x_1}\phi(s_2,x_2)\,dx_2-\int_{\alpha^-}^{\alpha^+}\partial_{x_1}\phi(s_1,x_2)\,dx_2,$$
hence the function $x_1\mapsto\int_{\alpha^-}^{\alpha^+}\partial_{x_1}\phi(x_1,x_2)\,dx_2$ is nondecreasing and has two limits~$\ell^{\pm}$ in~$[-\infty,+\infty]$ as $x_1\to\pm\infty$. One then infers from the boundedness of $\phi$ that $\ell^\pm=0$. Therefore,
$$\int_{\alpha^-}^{\alpha^+}\partial_{x_1}\phi(x_1,x_2)\,dx_2=0$$
for all $x_1\in\R$, which in turns implies that $\int\!\!\!\int_{(s_1,s_2)\times(\alpha^-,\alpha^+)}\Delta\phi=0$ for all $s_1<s_2$. From~\eqref{Deltaphi} and the positivity of $\phi$, one gets that $\partial_{x_1x_1}\psi=\partial_{x_1x_2}\psi=\partial_{x_2x_2}\psi=0$ in $\overline{\Omega}$. In particular, the function $\partial_{x_2}\psi$ is constant in $\overline{\Omega}$ and, since it vanishes on $\partial\Omega$, it is identically $0$ in $\overline{\Omega}$. Therefore, $\psi$ is a function of the variable~$x_1$ only and, since $\partial_{x_1x_1}\psi=0$ in $\overline{\Omega}$, there exist two real numbers $a$ and $b$ such that $\psi(x)=a\,x_1+b$ for all $x\in\overline{\Omega}$. The condition $\inf_{\overline{\Omega}}|\nabla\psi|>0$ gives that $a\neq0$, completing the proof.\hfill$\Box$

\begin{remark}
One uses the fact that $N=2$ to derive~\eqref{neumann}, that is, $\nabla\phi\cdot n=0$ on~$\partial\Omega$. This property does not hold in general in dimension $N\ge3$. For instance, let $D$ be the open Euclidean disk in $\R^2$ centered at the origin and with radius $R>0$, let $(r,\theta)$ be the usual polar coordinates in $\R^2$, and let $\psi$ be a $C^2(\R\times\overline{D},\R)$ harmonic function satisfying $\nabla\psi\cdot n=0$ on $\R\times\partial D$. In other words, $\partial_r\Psi(x_1,R,\theta)=0$ for all $(x_1,\theta)\in\R^2$, by calling~$\Psi(x_1,r,\theta)=\psi(x_1,r\cos\theta,r\sin\theta)$ for $(x_1,r,\theta)\in\R\times[0,R]\times\R$. Then the positive bounded function $\phi:=|\nabla\psi|^{-2}$ satisfies
$$(\nabla\phi\cdot n)(x_1,R\cos\theta,R\sin\theta)=\frac{2\,(\phi(x_1,R\cos\theta,R\sin\theta))^2\,(\partial_\theta\Psi(x_1,R,\theta))^2}{R^3}\ge0$$
for all $(x_1,\theta)\in\R^2$, that is, $\nabla\phi\cdot n\ge0$ on $\R\times\partial D$, and this nonnegative sign would make the arguments following~\eqref{neumann} ineffective.
\end{remark}


\end{document}